\documentclass[11pt]{amsart}
\usepackage[utf8]{inputenc}
\usepackage[T1]{fontenc}
\usepackage{a4wide}
\usepackage{mathpazo}
\usepackage{hyperref}
\usepackage{mathrsfs}
\usepackage{amssymb, amsmath}
\usepackage{enumitem}
\usepackage{xspace}
\usepackage[all]{xy}
\usepackage{bm}

\newcommand{\sC}{\mathscr C}

\newcommand{\NN}{\mathbf{N}}
\newcommand{\ZZ}{\mathbf{Z}}
\newcommand{\RR}{\mathbf{R}}

\newcommand{\sH}{\mathscr{H}}

\newcommand{\sN}{\mathscr{N}}
\newcommand{\sY}{\mathscr{Y}}

\newcommand{\sU}{\mathscr{U}}

\newcommand{\se}{\subseteq}

\newcommand{\fhi}{\varphi}
\newcommand{\lra}{\longrightarrow}
\newcommand{\inv}{^{-1}}

\newcommand{\da}[1]{{#1\!\downarrow}}
\newcommand{\ua}[1]{{#1\!\uparrow}}
\newcommand{\cha}[1]{\widehat{#1}}
\newcommand{\Ga}{G_{\mathrm{Arc}}}
\newcommand{\Gb}{G_{\Br}}
\newcommand{\Gc}{G_{\mathrm{Comp}}}
\newcommand{\Ge}{G_{\Ends}}
\newcommand{\Gr}{G_{\Reg}}

\DeclareMathOperator{\Aut}{Aut}
\DeclareMathOperator{\Homeo}{Homeo}

\DeclareMathOperator{\Ends}{Ends}
\DeclareMathOperator{\Free}{Free}
\DeclareMathOperator{\Sym}{Sym}
\DeclareMathOperator{\Alt}{Alt}
\DeclareMathOperator{\Reg}{Reg}
\DeclareMathOperator{\Br}{Br}

\DeclareMathOperator{\Fix}{Fix}
\DeclareMathOperator{\Stab}{Stab}

\theoremstyle{plain}
\newtheorem{thm}{Theorem}[section]
\newtheorem*{thm*}{Theorem}
\newtheorem{lem}[thm]{Lemma}
\newtheorem{prop}[thm]{Proposition}
\newtheorem{cor}[thm]{Corollary}
\theoremstyle{definition}
\newtheorem*{defn*}{Definition}
\newtheorem{defn}[thm]{Definition}
\newtheorem{example}[thm]{Example}
\newtheorem*{example*}{Example}
\newtheorem{rem}[thm]{Remark}

\newtheorem*{rem*}{Remark}
\begin{document}
\title{Structural properties of dendrite groups}
\author[B. Duchesne]{Bruno Duchesne}
\address{Institut Élie Cartan, Université de Lorraine, Nancy, France.}
\email{bruno.duchesne@univ-lorraine.fr}
\thanks{B.D. is supported in part by French projects ANR-14-CE25-0004 GAMME and ANR-16-CE40-0022-01 AGIRA}
\author[N. Monod]{Nicolas Monod}
\address{EPFL, 1015 Lausanne, Switzerland.}
\email{nicolas.monod@epfl.ch}
\begin{abstract}
Let $G$ be the homeomorphism group of a dendrite. We study the normal subgroups of $G$. For instance, there are uncountably many non-isomorphic such groups $G$ that are simple groups. Moreover, these groups can be chosen so that any isometric $G$-action on any metric space has a bounded orbit. In particular they have the fixed point property~(FH).
\end{abstract}
\maketitle
\setcounter{tocdepth}{1}
\renewcommand{\thesubsection}{{\thesection.\Alph{subsection}}}

\section{Introduction}
This article investigates the class of groups $G=\Homeo(X)$ that appear as homeomorphism group of a dendrite $X$, recalling that a \emph{dendrite} is a locally connected continuum without simple closed curves and that a continuum is a connected compact metrisable space.

Not much can be said unless we ask that the relationship between the group $G$ and the dendrite $X$ has some substance: for instance, there exists complicated dendrites with trivial homeomorphism group, and conversely the rather trivial ``star'' dendrite contains any countable group in its homeomorphism group.

\medskip
We shall therefore focus on \emph{dendro-minimal} dendrites, namely dendrites $X$ that do not admit a proper sub-dendrite $Y\se X$ invariant under $G=\Homeo(X)$.

Our main goal is to relate the properties of $G$, as an abstract group, with those of the topological space $X$. Inbetween these two worlds, we can also consider $G$ as a topological transformation group and hence as a Polish group.

\medskip
We start with some evidence that this whole entreprise has content:

\begin{thm}\label{thm:i:DS}
There is a family of $2^{\aleph_0}$ dendro-minimal dendrites such that the corresponding groups are pairwise non-isomorphic simple groups.
\end{thm}

We prove Theorem~\ref{thm:i:DS} with the explicit family of \emph{generalised Wa\.zewski dendrites} $D_S$, where $S$ is an arbitrary non-empty subset of $\{3, 4, \ldots, \infty\}$. These classical dendrites are characterised by the fact that the order of every branch point lies in $S$ and that such branch points are arcwise dense in $D_S$ for every order in $S$.

For the non-isomorphy statement of Theorem~\ref{thm:i:DS}, we shall prove that $\Homeo(D_S)$, as a group, determines the set $S$. This can also be deduced from Rubin's deep Reconstruction Theorem~\cite{Rubin89} but we provide a direct proof (see Section~\ref{sec:non-i} for a discussion).

\medskip

By contrast, the simplicity statement of Theorem~\ref{thm:i:DS} results from a general analysis of the normal subgroups of $G=\Homeo(X)$ for all dendro-minimal dendrites $X$. A first distinction to make is whether $X$ contains a \emph{free arc}, namely an embedded copy $I\cong [0,1]$ of an interval whose interior is open in $X$. This produces a copy in $G$ of the orientation-preserving group $\Homeo_+([0,1])$ and it turns out that $G$ can be decomposed as a permutational wreath product over the set $\Free(X)$ of all maximal free arcs, see Section~\ref{sec:free} for exact definitions for the following:

\begin{prop}\label{prop:i:semi:wreath}
Let $X$ be a dendro-minimal dendrite. Suppose that $X$ admits some free arc (but is not reduced to an arc). Then
$$\Homeo(X)\ \cong\ \prod_{I\in \Free(X)} \Homeo_+(I) \ \rtimes \Aut(\Free(X)).$$
\end{prop}

For the analysis of the normal subgroups of $G=\Homeo(X)$ with $X$ general, consider the normal subgroup $\Ge \lhd G$ generated by all stabilisers of end points of $X$, and likewise $\Gb \lhd G$ the group generated by all stabilisers of branch points.

\begin{thm}\label{thm:i:simple}
Let $X$ be any dendro-minimal dendrite not reduced to an arc. Then $\Ge = G$.

Moreover, any subgroup normalised by $\Gb$ either contains $\Gb$ or fixes $\Ends(X)$ pointwise.
\end{thm}

This will imply the following simplicity statements.

\begin{cor}\label{cor:i:simple}
Let $X$ be any dendro-minimal dendrite.

If $X$ has no free arc, then $\Gb$ is a simple group.

If $X$ admits some free arc (but is not reduced to it), then the image of $\Gb$ in $\Aut(\Free(X))$ is a simple group.
\end{cor}

The approach of the present article is in a sense opposite to our approach in~\cite{DM_dendrites}, where we studied obstructions for group actions on dendrites. There, we took the external viewpoint of representing some other group into $G$. Several of these results pointed to analogies with negative curvature, especially with results for isometric actions on trees. However, to stress how different the topological setting of dendrites can be, we proposed the following problem:

\itshape Find a Kazhdan group with a non-elementary action on a dendrite\upshape.

\noindent
Such a group would have a fixed point for any isometric action on a tree, or on a complete $\RR$-tree.

\medskip
Currently we do not have such an example if \emph{Kazhdan group} is understood in the context of abstract groups, which must then be finitely generated. We do however have many examples with \emph{property~(FH)}, i.e.\ the fixed-point property for isometric actions on Hilbert spaces. The latter property is precisely the consequence of Kazhdan's property which forbids actions on trees and on $\RR$-trees. It is moreover equivalent to Kazhdan's property for countable groups.

In fact the next result is much stronger than property~(FH) and shows that the dendrites $D_S$ are strongly allergic to the very idea of any metric structure.

\begin{thm}\label{thm:i:OB}
Let $S$ be a non-empty subset of $\{3, 4, \ldots, \infty\}$. Then any isometric action of $\Homeo(D_S)$ on any metric space has a bounded orbit.

In particular, the group $\Homeo(D_S)$ has property~(FH).
\end{thm}

Another way to come close to an answer to the problem of Kazhdan groups is to consider $\Homeo(X)$ with its natural topology as a homeomorphism group. We can then use a result of Evans--Tsankov~\cite{Evans-Tsankov} and deduce that certain dendrite groups even have the \emph{strong} version of Kazhdan's property, namely admit a finite Kazhdan set.

\begin{thm}\label{thm:i:T}
Let $S$ be a finite non-empty subset of $\{3, 4, \ldots, \infty\}$.

Then the Polish group $\Homeo(D_S)$ has the strong Kazhdan property~(T).
\end{thm}

%

\bigskip
\subsection*{Location of the proofs}
The simplicity proof for Theorem~\ref{thm:i:DS} is completed in Corollary~\ref{cor:DS:simple} and the non-isomorphy is Corollary~\ref{cor:non-iso}. Proposition~\ref{prop:i:semi:wreath} is contained in Theorem~\ref{thm:semi:wreath}. Theorem~\ref{thm:i:simple} follows from Theorems~\ref{thm:fix:incl} and~\ref{thm:simple}. Corollary~\ref{cor:i:simple} is then deduced as Corollaries~\ref{cor:simple} and~\ref{cor:simple:free}. Theorem~\ref{thm:i:OB} occurs as Corollary~\ref{cor:OB} below, and Theorem~\ref{thm:i:T} as Corollary~\ref{cor:DS:T}.

\tableofcontents

\section{Preliminaries}
A general reference for dendrites is~\S10 in~\cite{Nadler}. We now recall some of the background but refer to~\cite{Nadler} and to~\cite{DM_dendrites} for the points not justified below.

\subsection{Dendrites}
In the case of a dendrite $X$, the Menger--Urysohn \emph{order} of a point $x\in X$ is simply the cardinality of the set of components of $X\setminus \{x\}$. The points of order one, two or~$\geq 3$ form respectively the sets $\Ends(X)$ of \emph{end points}, $\Reg(X)$ of \emph{regular points} and $\Br(X)$ of \emph{branch points}. For $3\leq n \leq \infty$, we write $\Br_n(X)$ for the set of points of order~$n$. The set $\Br(X)$ is always countable, $\Reg(X)$ is arcwise dense and uncountable, and $\Ends(X)$ is non-empty (assuming $X$ is not reduced to a point).

\medskip
Every non-empty closed closed connected subset $Y\se X$ is itself a dendrite; therefore we also refer to such sets as \emph{sub-dendrites}. There exists then a canonical continuous retraction $X\to Y$ called the \emph{first-point map}. For an arbitrary non-empty subset $Y\subset X$, there is a unique minimal closed connected sub-dendrite of $X$ containing $Y$; we denote it by $[Y]$. We write $[x,y]$ for the unique topological arc in $X$ connecting two points $x$ and $y$. A \emph{free arc} is an arc, not reduced to a point, whose interior is open in $X$. Since $X$ is a dendrite, it is equivalent to ask that the interior of the arc contains no branch point. Any free arc is contained in a maximal free arc; we denote by $\Free(X)$ the set of all maximal free arcs of $X$, which is a countable set. Two maximal free arcs meet at most at a common extremity, which is then a branch point. The following are equivalent, provided $X$ is not reduced to a point: $X$ has no free arc; $\Br(X)$ is dense in $X$; $\Ends(X)$ is dense in $X$.

\medskip
Any continuous self-map of a dendrite has a fixed point. A non-trivial arc $I=[x,y]$ in $X$ is \emph{austro-boreal} for a given homeomorphism $g\in\Homeo(X)$ if the intersection of the fixed-point set of $g$ in $X$ with $I$ is $\{x,y\}$. In that case, $g$ preserves $I$ and its action on $I\setminus \{x,y\}$ is conjugated to a translation action on~$\RR$.

\medskip
The following lemma can be deduced e.g.\ from~\cite[\S12]{Kuratowski-Whyburn}.

\begin{lem}\label{lem:cut:small}
Let $d$ be a metric inducing the topology of the dendrite $X$. For every $\epsilon>0$, there is a finite set $F$ such that the components of $X\setminus F$ have $d$-diameter less than~$\epsilon$.\qed
\end{lem}

It could alternatively be deduced from the following special case of~\cite[V.2.6]{Whyburn_book}, which we shall also need below.

\begin{lem}\label{lem:null}
Any sequence of disjoint connected subsets of a dendrite is a \emph{null-sequence}, i.e.\ the diameter of its members tends to zero for any compatible metric.\qed
\end{lem}

\begin{rem}\label{rem:F:dense}
The set $F$ of Lemma~\ref{lem:cut:small} can moreover be taken to belong to any given arcwise dense subset of $X$. Indeed, if $F_0$ is as in the lemma but for some $\epsilon_0 < \epsilon/2$, then any set $F$ containing a point in $[x,y]$ for all distinct $x,y\in F_0$ will satisfy the condition for~$\epsilon$. On the other hand, one cannot restrict $F$ to lie in a subset that is merely dense, as evidenced by dendrites with a dense set of ends.
\end{rem}

An action of group $G$ on a dendrite is \emph{elementary} if there a fixed point or an invariant pair of points. An action is \emph{dendro-minimal} if there is no invariant proper sub-dendrite. See~\cite[\S3]{DM_dendrites} for characterizations of elementarity and~\cite[\S4]{DM_dendrites} for relations between dendro-minimality and the usual minimality for actions on compact spaces.

\subsection{Topologies on $\Homeo(X)$}
A natural topology on the homeomorphism group $\Homeo(X)$ of a dendrite $X$ is the topology of uniform convergence. With this topology, $\Homeo(X)$ is a Polish group (see e.g.\ \cite[\S 9.B.9]{Kechris_book}). When $X$ has no free arcs, a homeomorphism is completely determined by its action on $\Br(X)$ since $\Br(X)$ in dense is $X$. In particular, there is a faithful representation $\Homeo(X)\to\Sym(\Br(X))$, where $\Sym$ is used to denote the group of all permutations of a given set. Since the set of branch points $\Br(X)$ is countable, $\Sym(\Br(X))$ is a Polish group for the pointwise convergence. This gives a priori two topologies on $\Homeo(X)$; in fact they coincide.

\begin{prop}\label{prop:topol}
If $X$ is a dendrite without free arcs, the homomorphism $\Homeo(X)\to\Sym(\Br(X))$ is a homeomorphism onto its image.

In particular, $\Homeo(X)$ embeds as a closed subgroup of $\Sym(\Br(X))$.
\end{prop}

Since $\Homeo(X)$ is a closed subgroup of $\Sym(\Br(X))$, it implies in particular that it is a totally disconnected group; more precisely, its open subgroups form a neighbourhood system for the identity.

\begin{proof} 
We first show that the topology induced from $\Sym(\Br(X))$ is finer than the uniform topology. The pointwise stabiliser in $\Homeo(X)$ of a finite set $F\se \Br(X)$ as in Lemma~\ref{lem:cut:small} permutes the components of $X\setminus F$. Let $F'\supseteq F$ be a finite set containing also a point of each of these components. Then the pointwise stabiliser of $F'$ preserves each connected component of $X\setminus F$ and therefore is $\epsilon$-close to the identity in the uniform metric.

Conversely, let $(g_n)$ be a sequence in $\Homeo(X)$ converging to the identity in the uniform topology. For any $b\in\Br(X)$, we choose $x_1, x_2, x_3$ in distinct components of $X\setminus\{b\}$. Since the connected components of $X\setminus\{b\}$ are open, we know that $g_n(x_i)$ and $x_i$ belongs to the same component for $n$ large enough. Since $b$ is the only point in $[y_1, y_2]\cap [y_2, y_3]\cap [y_3, y_1]$ for any choice of elements $y_i$ in the corresponding components, $g_n(b)=b$ follows.

Finally, the image of $\Homeo(X)$ in $\Sym(\Br(X))$ is closed because it is Polish, see Exercise~9.6 in~\cite{Kechris_book}.
\end{proof}

\subsection{Separating subsets}
The following elementary disjunction is Lemma~4.3 in~\cite{DM_dendrites}.

\begin{lem}\label{lem:dis}
Let $G$ be a group with a dendro-minimal action on a dendrite $X$. For any proper sub-dendrite $Y\se X$ there is $g\in G$ such that $gY\cap Y=\varnothing$.\qed
\end{lem}

The proof given in~\cite{DM_dendrites} was as follows: if $gY\cap Y$ were never empty, then by a basic Helly theorem the intersection of $gY$ over all $g\in G$ would be non-empty, contradicting dendro-minimality.

\medskip

We shall use the following strengthening of Lemma~\ref{lem:dis}.

\begin{lem}\label{lem:dis2}
Let $G$ be a group with a dendro-minimal action on a dendrite $X$.

For any proper sub-dendrites $Y, Y'\se X$ there is $g\in G$ such that $gY \cap Y' = \varnothing$.
\end{lem}

\begin{proof}
We can assume that $Y$ meets $Y'$ since otherwise $g=e$ will do. Therefore, $Y\cup Y'$ is a sub-dendrite of $X$. If this sub-dendrite is proper, then Lemma~\ref{lem:dis} applied to $Y\cup Y'$ provides $g$ with $g(Y\cup Y') \cap (Y\cup Y')=\varnothing$ and a fortiori $gY \cap Y' = \varnothing$. Therefore we can assume $Y \cup Y' = X$. Now Lemma~\ref{lem:dis} yields $g_1$ with $g_1Y\cap Y=\varnothing$ and hence $g_1Y \se Y'$. One last application of that lemma gives us $g_2$ with $g_2Y'\cap Y'=\varnothing$. We conclude that $g=g_2 g_1$ satisfies $g Y \se g_2 Y' \se X\setminus Y'$ as was to be shown.
\end{proof}

A minor variation of Lemma~\ref{lem:dis2} is as follows.

\begin{cor}\label{cor:dis3}
Let $G$ be a group with a dendro-minimal action on a dendrite $X$.

For any connected set $Z\se X$ that is not dense, any point $p\in X$ and any component $C$ of $X\setminus\{p\}$, there is $g\in G$ such that $g Z \se C$.
\end{cor}

\begin{proof}
The complement $Y'=X \setminus C$ is a proper sub-dendrite and we can therefore apply Lemma~\ref{lem:dis2} with $Y = \overline{Z}$.
\end{proof}

We record a more technical variant of Corollary~\ref{cor:dis3}.

\begin{lem}\label{lem:dis+}
Let $G$ be a group acting dendro-minimally on a dendrite $X$. Let $I\se X$ be an arc and $b$ be a branch point in the interior of $I$. Let $x\in X$ and let $Y$ be a component of $X\setminus\{x\}$. 

Then there is $g\in G$ with $g I\se Y$ such that the image of $x$ under the first-point map to $gI$ is $gb$.
\end{lem}

\begin{proof}
Upon replacing $g$ by $g^{-1}$, the task is to find $g\in G$ such $I\se gY$ and $b$ is the retraction of $gx$ on $I$. Let $p$ be a regular point on an arc branching off from $I$ at $b$, let $C$ be the component of $X\setminus \{p\}$ not containing $b$ and denote by $Z$ the dendrite $\{x\} \cup (X\setminus Y)$. Then it suffices to have $gZ \se C$, which is possible by Corollary~\ref{cor:dis3}.
\end{proof}

\subsection{Patchwork}
The following patchwork for homeomorphisms fails already if we replace dendrites by a disc.
 
\begin{lem}\label{lem:patchwork}
Let $\sU$ be a family of disjoint open connected subsets of a dendrite $X$ and let $(f_U)_{U\in \sU}$ be a family of homeomorphisms $f_U \in\Homeo(U)$ for $U\in\sU$. Suppose that each $f_U$ can be extended continuously to the closure $\overline U$ by the identity on the boundary $\overline U\setminus U$.

Then the map $f\colon X\to X$ given by $f_U$ on each $U\in \sU$ and the identity everywhere else is a homeomorphism.
\end{lem}

\begin{proof}
By compactness, it suffices to prove that there cannot be a sequence $x_n\to x$ in $X$ such that $f(x_n)$ converges to some $y\neq f(x)$. If there is such a sequence, then each $U\in\sU$ contains $x_n$ for at most finitely many indices $n$. Moreover, upon extracting we can assume that each $x_n$ belongs to some $U_n\in \sU$. By the previous observation, we can assume that all $U_n$ are distinct and hence they form a null-sequence by Lemma~\ref{lem:null}. Since $f(x_n) = f_{U_n}(x_n)$ belongs also to $U_n$, we conclude that $f(x_n)$ converges to $f(x)$, contrary to our assumption.
\end{proof}

\subsection{Uniform structure}
Like any compact space, a dendrite admits a unique compatible uniform structure.

\begin{lem}\label{lem:unif}
For any dendrite $X$, the sets
$$U_F = \{(x,y)\in X^2 : |[x,y]\cap F | \leq 1\}$$
form a basis of the uniform structure as $F$ ranges over all finite subsets of $X$. Moreover, one can restrict $F$ to belong to any arcwise dense subset of $X$.
\end{lem}

\begin{proof}
These sets belong to the uniformity associated to the compact topology, i.e.\ they are neighbourhoods of the diagonal in $X\times X$: this follows from the fact that $X$ is locally arcwise connected (like any Peano continuum~\cite[8.25]{Nadler}). In oder to prove that they form a basis, choose a compatible metric $d$ on $X$ and $\epsilon>0$. Let $F$ be a finite set as provided by Lemma~\ref{lem:cut:small}; we can assume that $F$ is in a given arcwise dense subset by Remark~\ref{rem:F:dense}. If now $(x,y)\in U_F$, then $[x,y]\setminus F$ has at most three components and thus $d(x,y)\leq 3\epsilon$.
\end{proof}

\subsection{Permutational wreath products}
We begin with a result about cartesian products.

\begin{lem}\label{lem:fixprodgen}
Suppose that the direct product $P=\prod_{N\in\sN} N$ of a family $\sN$ of groups acts on a dendrite. If each $N\in\sN$ has fixed points, then so does $P$.
\end{lem}

\begin{proof}
By a compactness argument, we can choose a minimal $P$-invariant sub-dendrite $Y\se X$. Each $N$ has fixed points in $Y$ since the first-point map $X\to Y$ is $P$-equivariant by its naturality. Since $P$ normalises $N$ it follows that $[Y\cap\Fix(N)]$ is a $P$-invariant sub-dendrite and hence $[Y\cap\Fix(N)]=Y$. It follow that $\Ends(Y)$ is contained in $\Fix(N)$ by~\cite[Lemma 2.3]{DM_dendrites}. Since this holds for each $N\in\sN$, we conclude that $P$ fixes $\Ends(Y)$.
\end{proof}

\begin{prop}\label{prop:perm}
Let $N$ be a group without index two subgroup and let $\sC$ be a set of cardinality $3\leq |\sC|\leq \aleph_0$. Then every action of the full permutational wreath product
$$H=N^\sC \rtimes \Sym(\sC)$$
on a dendrite has a fixed point.
\end{prop}

\begin{proof}
Let $X$ be a dendrite with an $H$-action. Since $N^\sC$ is a product, one of its factors must act elementarily on $X$ by Corollary~4.6 in~\cite{DM_dendrites} (that reference considers products of two factors, but the general case follows at once). Since $N$ has no index two subgroup, elementarity means that this factor fixes a point in $X$. Since all factors of $N^\sC$ are conjugated in $H$, they all fix some point. Therefore, $N^\sC$ has a fixed point by Lemma~\ref{lem:fixprodgen}.

We consider first the case where $\sC$ has finite size $n\geq 3$. Then $H$ has a finite orbit and hence acts elementarily. If $H$ does not fix a point, it preserves some arc $I\se X$ and admits an index two subgroup $H'$ preserving the orientation of $I$. We claim that $H'$ is $N^\sC \rtimes \Alt(n)$. Since $N$ has no index two subgroup, it suffices to check that $\Alt(n)$ is the only index two subgroup of $\Sym(n)$, or equivalently that it has no index two subgroup itself. Indeed, for $n\neq 4$ it is simple (of order~$\neq 2$) and for $n=4$ its only proper normal subgroup has index three. This proves the claim.

Since the group of orientation-preserving homeomorphisms of an interval contains no non-trivial finite subgroup, the image of $\Alt(n)$ in $\Homeo_+(I)$ is trivial. There is an element $\sigma\in\Sym(n)$ reversing the orientation of $I$; note that $\sigma$ has a necessarily unique fixed point $p$ in the interior of $I$. It suffices to prove that every factor $N_1\cong N$ of $N^\sC$ fixes $p$. If not, then for some $g\in N_1$ we have $g p > p$ with respect to a choice of an orientation of $I$. Then $\sigma g \sigma\inv p=\sigma g p < p$. On the other hand, $\Alt(n)$ acts transitively on $\sC$ since $n\geq 3$. Therefore there is $\tau\in \Alt(n)$ with $\tau g \tau\inv =\sigma g \sigma\inv$. It follows $\tau g \tau\inv p< p$, which contradicts $g p > p$ since $\tau$ acts trivially on $I$.

We now consider the case where $\sC$ is infinite countable. Denote by $\Sym_{\mathrm f}(\sC)$ the normal subgroup of finitely supported permutations of $\sC$ and by $\Alt_{\mathrm f}(\sC)$ its alternating subgroup. Since $\Sym_{\mathrm f}(\sC)$ is an amenable group and since $N^\sC$ fixes a point, Lemma~6.1 in~\cite{DM_dendrites} implies that $N^\sC\rtimes \Sym_{\mathrm f}(\sC)$ acts elementarily. Let thus $I\se X$ be an invariant arc; note that $N^\sC$ fixes a point of $I$, namely the image of any fixed point under the first-point map $X\to I$. As above, $\Alt_{\mathrm f}(\sC)$ acts trivially on $I$, being an increasing union of $\Alt(n)$. Therefore, $N^\sC\rtimes \Alt_{\mathrm f}(\sC)$ fixes a point. The non-empty set $K\se X$ of points fixed by $N^\sC\rtimes \Alt_{\mathrm f}(\sC)$ is closed, and therefore it is a compact metrisable space. It is $H$-invariant and the $H$-action on $K$ descends to an action of $\Sym(\sC)/\Alt_{\mathrm f}(\sC)$. If we endow $\Sym(\sC)$ with its Polish topology, the subgroup $\Alt_{\mathrm f}(\sC)$ is dense. However, Kechris--Rosendal have proved that any action of $\Sym(\sC)$ on a compact metrisable space is continuous for this topology, see Theorem~1.10 in~\cite{KEchris-Rosendal}. Therefore, the $\Sym(\sC)$ action on $K$ is trivial and $H$ fixes every point of $K$.
\end{proof}

We record the following for later use.

\begin{lem}\label{lem:no:ind2}
Let $N$ be any group and let $\sC$ be an infinite countable set. Then the group
$$H=N^\sC \rtimes \Sym(\sC)$$
has no index two subgroup.
\end{lem}

\begin{proof}
Schreier--Ulam proved in~\cite[\S1.7]{Schreier-Ulam} that the only proper normal subgroups of $\Sym(\sC)$ are $\Sym_{\mathrm f}(\sC)$ and $\Alt_{\mathrm f}(\sC)$. Therefore, an index two subgroup of $H$ would have to be of the form $M\rtimes \Sym(\sC)$, where $M\lhd N^\sC$ is of index two and $\Sym(\sC)$-invariant. In particular, for every $y\in N^\sC$ and every $\sigma\in\Sym(\sC)$, we have $y^\sigma  y^{-1} \in M$. Let now $z$ be an arbitrary element of $N^\sC$. It suffices to express $z$ as $y^\sigma  y^{-1}$. To this end, we take $\sigma$ to be the shift under identifying $\sC$ with $\ZZ$. We can now define $y=(y_n)_{n\in \ZZ}$ inductively by taking $y_0$ arbitrary and setting $y_n = z_{n-1} y_{n-1}$ for $n>0$ and $y_n = z_n^{-1} y_{n+1}$ for $n<0$.
\end{proof}

\subsection{The generalised Wa\.zewski dendrites $D_S$}
Let $S$ be a non-empty subset of $\{3,4,\dots,\infty\}$; here $\infty$, standing for ``not finite'', will in fact be the cardinal~$\aleph_0$. There exists a dendrite, not reduced to a point, whose branch points all have orders in $S$ and with the following property: for every $n\in S$, every arc contains an element of order $n$. It follows that $\Br_n(D_S)$ is arc-wise dense. The existence of such a dendrite can be proved for instance by an inverse limit construction following the procedure given in~\cite[10.37]{Nadler} for the special case $S=\{\infty\}$. In fact, such a dendrite is unique up to homeomorphism by Theorem~6.2 in~\cite{Charatonik-Dilks}. It is called the \emph{generalised Wa\.zewski dendrite} $D_S$. The uniqueness implies:

\begin{lem}\label{lem:sub:DS}
The closure of any non-empty connected open subset of $D_S$ is homeomorphic to $D_S$.\qed
\end{lem}

We point out that $\Ends(D_S)$ is uncountable. Indeed, it has no isolated points and it is Polish since it is a dense $G_\delta$ in $D_S$. The latter is the case because its complement can be written as the union of all arcs $[x,y]$ where $x,y$ range over the countable set $\Br(D_S)$.

\section{Normal subgroups from free arcs}\label{sec:free}
There are several natural normal subgroups of $\Homeo(X)$ related to free arcs: we define
$$\Fix(\Free(X))\ <\ \Stab_+(\Free(X))\ <\ \Stab(\Free(X))$$
to consist of the homeomorphisms for which every $I\in\Free(X)$ is fixed pointwise, respectively preserved together with its orientation, respectively just preserved.

Furthermore, we denote by $\Fix(\Ends(X))$ the normal subgroup of $\Homeo(X)$ consisting of the elements that fix every end of $X$.

\begin{prop}\label{prop:fix:stab}
Let $X$ be a dendrite admitting some free arc.

Then the group $\Fix(\Ends(X))$ is contained in $\Stab_+(\Free(X))$ and we have a decomposition
$$\Stab_+(\Free(X))\ =\ \Fix(\Free(X))\ \times\ \Fix(\Ends(X)) .$$
Moreover, the restriction maps to $\Homeo(I)$ induce an isomorphism
$$\Fix(\Ends(X))\ \cong \ \prod_{I\in \Free(X)} \Homeo_+(I).$$
\end{prop}

\begin{proof}
We claim that $\Fix(\Ends(X))$ fixes every branch point $p$ of $X$. Indeed, there exists three arcs meeting only at $p$; by a maximality argument, the other extremity of each of those can be taken to be an end of $X$. Any homeomorphism fixing these three ends fixes $p$, whence the claim. Since any extremity of a maximal free arc is either an end, or a branch point, or a limit of branch points, the claim shows already that $\Fix(\Ends(X))$ is contained in $\Stab_+(\Free(X))$.

Next, we claim that the two normal subgroups $\Fix(\Ends(X))$ and $\Fix(\Free(X))$ intersect trivially; hence they commute and generate their direct product. For this we need to show that any homeomorphism in their intersection fixes every regular point of $X$. This is the case by the previous claim since any regular point either is a limit of branch points or is contained in a free arc.

Finally, we consider the restriction morphism from $\Stab_+(\Free(X))$ to $\prod_{I\in \Free(X)} \Homeo_+(I)$, whose kernel is exactly $\Fix(\Free(X))$; in particular it is injective on $\Fix(\Ends(X))$. On the other hand this map is surjective when restricted to $\Fix(\Ends(X))$ because of Lemma~\ref{lem:patchwork} applied to the interiors of all $I\in \Free(X)$. Both isomorphisms of Proposition~\ref{prop:fix:stab} follow.
\end{proof}

\begin{prop}\label{prop:free:dense}
Let $X$ be a dendrite admitting some free arc.

If $X$ is dendro-minimal, then the union of all maximal free arcs is dense in $X$.
\end{prop}

\begin{proof}
It suffices to prove that any branch point $p\in X$ is a limit of points belonging to free arcs. Choose $q \neq p$; we can assume that $[p,q]$ contains a sequence $(p_n)$ of distinct branch points $p_n$ converging to $p$ because otherwise $p$ would bound a free arc in $[p,q]$. Let $c_n$ be a component of $X\setminus \{p_n\}$ not containing $p$ nor $q$. Then all $c_n$ are pairwise disjoint and therefore form a null-sequence by Lemma~\ref{lem:null}. It remains only to observe that each $c_n$ contains a free arc; this follows from Corollary~\ref{cor:dis3}, choosing $Z$ to be some free arc.
\end{proof}

\begin{cor}\label{cor:fix:stab}
Let $X$ be a dendrite admitting some free arc. If $X$ is dendro-minimal, then 
$$\Stab_+(\Free(X))\ =\  \Fix(\Ends(X))\ \cong \ \prod_{I\in \Free(X)} \Homeo_+(I).$$
Furthermore,
$$\Stab(\Free(X))\ =\ \Stab_+(\Free(X))$$
unless $X$ is reduced to an arc.
\end{cor}

\begin{proof}
By Proposition~\ref{prop:free:dense}, the group $\Fix(\Free(X))$ is trivial. Therefore, the first statement follows from Proposition~\ref{prop:fix:stab}.

It remains to justify that every element $g$ of $\Stab(\Free(X))$ preserves the orientation of any $I\in\Free(X)$ unless $X=I$. Thus we can assume that at least one extremity $p$ of $I$ is not an end of $X$. Therefore, there is a component $c$ of $X\setminus \{p\}$ which does not meet $I$. By Corollary~\ref{cor:dis3}, $c$ contains a maximal free arc $J$, namely an image of $I$. Since $g$ preserves $J$ and $I$, it cannot permute the extremities of $I$.
\end{proof}

Define $\Aut(\Free(X))$ to be the image of $\Homeo(X)$ in the permutation group of the countable set $\Free(X)$; in other words,
$$\Aut(\Free(X)) \ =\ \Homeo(X) / \Stab(\Free(X)).$$

\begin{thm}\label{thm:semi:wreath}
Let $X$ be a dendrite admitting some free arc, but not reduced to an arc.

If $X$ is dendro-minimal, then $\Homeo(X)$ admits a semi-direct product decomposition
$$\Homeo(X)\ \cong\ \Fix(\Ends(X))\ \rtimes \Aut(\Free(X))$$
and hence it is the permutational wreath product
$$\Homeo(X)\ \cong\ \prod_{I\in \Free(X)} \Homeo_+(I) \ \rtimes \Aut(\Free(X)).$$
\end{thm}

\begin{proof}
We choose for each $I\in \Free(X)$ a homeomorphism $\alpha_I\colon [0,1]\to I$. We further choose an involution $\tau$ of $[0,1]$ that exchanges the end-points.

If now $g$ is an arbitrary element of $\Homeo(X)$, we  define for $I\in \Free(X)$ the involution $\tau_I^g$ of $[0,1]$ to be trivial if $g\alpha_I$ and $\alpha_{gI}$ give the same orientation to $gI$ and to be $\tau$ otherwise. This can be written more obscurely as $\tau_I^g = \tau^{\alpha_{gI}\inv g \alpha_I(0)}$; in any case, we have the cocycle relation $\tau_I^{g_1 g_2} = \tau_{g_2 I}^{g_1} \tau_I^{g_2}$. We define furthermore the map $f_I^g$ on $I$ by $f_I^g= \alpha_I\circ \tau_I^g\circ \alpha_{gI}\inv\circ g|_I$. Then $f_I^g$ is in $\Homeo_+(I)$. Therefore, Proposition~\ref{prop:fix:stab} ensures that there is $f^g\in \Fix(\Ends(X))$ with $f^g|_I= f_I^g$ for all $I\in \Free(X)$. We can therefore define an element $\overline g\in\Homeo(X)$ by $\overline g = g \circ (f^g)\inv$.

The cocycle relation for $\tau_I^g$ shows that the assignment $g\mapsto \overline g$ is a group homomorphism. More precisely, this holds by construction when we evaluate on any point in a free arc and thus everywhere since the union of all free arcs is dense by Proposition~\ref{prop:free:dense}.

This homomorphism is trivial on $\Stab(\Free(X))$. Indeed, any $g$ in that group preserves the orientation of every free arc since $X$ is not an arc, see Corollary~\ref{cor:fix:stab}. Therefore the construction of $f_I^g$ shows that $f_I^g=g|_I$, hence $\overline g$ is trivial. In conclusion, the morphism $g\mapsto \overline g$ provides a section from  $\Aut(\Free(X))$ to $\Homeo(X)$ because $f^g$ is in  $\Fix(\Ends(X)$ which coincides with $\Stab(\Free(X))$ by Corollary~\ref{cor:fix:stab}.

Finally, the wreath product statement follows from the naturality of the isomorphism between $\Fix(\Ends(X))$ and $\prod_{I\in \Free(X)} \Homeo_+(I)$.
\end{proof}

\section{Normal subgroups from stabilisers}
Let $X$ be an arbitrary dendrite. A priori there is a number of natural normal subgroups of $\Homeo(X)$ associated to point stabilisers. Just to keep the arguments below more legible, we temporarily write $G=\Homeo(X)$ and introduce the following notation: $\Gb$, $\Ge$ and $\Gr$ are the subgroups of $G$ generated by all stabilisers of branch points, respectively of end points and regular points. Similarly, $\Ga$ denotes the subgroup generated by all pointwise stabilisers of non-trivial arcs, and $\Gc$ by all pointwise stabilisers of connected components of the complement of points.

All these are normal subgroups of $G$.

\begin{thm}\label{thm:fix:incl}
Let $X$ be any dendro-minimal dendrite not reduced to an arc. Then
$$\Ga = \Gb = \Gc < \Gr < \Ge =G. $$
\end{thm}

In fact, the first two equalities above are basic and more general:

\begin{lem}\label{lem:fix:incl}
For any dendrite not reduced to a point, $\Ga = \Gb = \Gc <  \Ge\cap \Gr$.
\end{lem}

\begin{proof}[Proof of Lemma~\ref{lem:fix:incl}]
By definition, we have $\Gc <  \Ga\cap \Gb\cap \Ge\cap \Gr$. We first claim that  $\Gb < \Gc$. Let thus $g\in G$ fix a branch point $x$. We define a homeomorphism $h$ that coincides with $g$ on some component $C$ of $X\setminus\{x\}$, that coincides with $g^{-1}$ on $g C$ and is the identity elsewhere. Since $X\setminus\{x\}$ has at least three connected components, $h\in \Gc$. Now, $g=(gh^{-1})h$ and $gh^{-1}$ fixes $C$ pointwise. Therefore, $g\in \Gc$. This proves the first claim and hence  $\Gb = \Gc$.

Our second claim is that any element $g$ fixing a regular point $r$ without exchanging the two components $C_1, C_2$ of $X\setminus\{r\}$ belongs to $\Gc$. Indeed, $g$ can then be written $g=g_1g_2$ where $g_i$ fixes pointwise $C_i$ and acts as $g$ does on the other component.

This second claim implies in particular $\Ga< \Gc$ and thus  $\Ga= \Gc$.
\end{proof}

\begin{proof}[Proof of Theorem~\ref{thm:fix:incl}]
We claim that any element $g\in G$ which admits an austro-boreal arc belongs to $\Ge$. Indeed, this is the case if an extremity of such an arc is an end of $X$, but also if it is a branch point since $\Gb <  \Ge$ by Lemma~\ref{lem:fix:incl}. If on the other hand the extremities are regular points, then $g$ cannot exchange the corresponding components since it preserves the arc. Thus we conclude again $g\in\Gc<\Ge$, this time by the second claim of the proof of Lemma~\ref{lem:fix:incl}.

Our main goal is to prove  $\Gr < \Ge$. It suffices to consider an element $g\in G$ which fixes a regular point $r$ and exchanges the components $C_1, C_2$ of $X\setminus\{r\}$. By Theorem~10.5 in~\cite{DM_dendrites}, $G$ contains an element $h$ admitting an austro-boreal arc $I$. By Lemma~\ref{lem:dis+}, we can assume, after conjugating $h$, that $I$ lies in $C_2$ and that the first-point map sends $r$ to some branch point $b$ in the interior of $I$.

We shall now prove that $hg \in \Ge$, which implies indeed $g\in \Ge$ since $h\in \Ge$ by the above claim. We know that $hg$ fixes some point $p\in X$. Given what we have established so far, the only problematic case is if $p$ is a regular point and if moreover $hg$ permutes the components of $X\setminus\{p\}$. We shall prove by contradiction that this cannot happen.

Indeed, in that case $p$ is the only fixed point of $hg$ and, moreover, for every $x\in X$ the arc $[x, hg x]$ contains $p$. We first apply this observation to $x=h b \in I$. Then the arc $[x, hg x]$ starts from $hb$, crosses $hr$ and ends in $h C_1$. Our assumption on the respective position of $I$ and $r$ shows that this entire arc projects to $hb$ under the first-point map $X\to I$. We now apply the observation to $x=(hg)\inv b \in C_1$. This time the arc $[x, hg x]$ starts in $C_1$, crosses $r$ and ends at $b$. Therefore, this arc projects  to $b$ under the first-point map $X\to I$. It follows that the two arcs do not intersect, a contradiction.

At this point we know that both $\Gb$ and $\Gr$ are contained in $\Ge$; since any homeomorphism of $X$ must fix some point, it follows that $G=\Ge$.
\end{proof}

\begin{lem}\label{lem:dm}
Let $X$ be a dendrite and let $N$, $H$ be subgroups of $\Homeo(X)$. 

Suppose that $H$ acts dendro-minimally on $X$ and normalises $N$. Then either $N$ acts dendro-minimally on $X$ or $N$ fixes $\Ends(X)$ pointwise.

In particular, if $X$ has no free arc, then any non-trivial normal subgroup of a dendro-minimal subgroup of $\Homeo(X)$ is dendro-minimal as well.
\end{lem}
 
\begin{proof}
If $N$ has no fixed point in $X$, then there is a \emph{unique} minimal $N$-invariant sub-dendrite $Y$ in $X$, see Remark~4.2 in~\cite{DM_dendrites}. Since $H$ normalises $N$, it preserves $Y$ and thus $Y=X$.

If on the other hand $N$ has a fixed point $x\in X$, then the $H$-orbit of $x$ consists of $N$-fixed points, and hence so does its closure. This closure contains $\Ends(X)$ by Lemma~4.4 in~\cite{DM_dendrites}.

The additional statement follows from the fact that $\Ends(X)$ is dense in $X$ if $X$ has no free arc (see e.g.~\cite[4.7]{DM_dendrites}).
\end{proof}

\begin{thm}\label{thm:simple}
Let $X$ be any dendro-minimal dendrite not reduced to an arc. Let $N<G=\Homeo(X)$ be any subgroup normalised by $\Gb$.

Then either $N$ contains $\Gb$ or $N$ fixes $\Ends(X)$ pointwise.
\end{thm}

If $X$ has no free arc, we deduce:

\begin{cor}\label{cor:simple}
Let $X$ be any dendro-minimal dendrite without free arc.

Then $\Gb$ is a simple group.\qed
\end{cor}

If on the contrary $X$ admits some free arc, then we appeal to Theorem~\ref{thm:semi:wreath} and deduce:

\begin{cor}\label{cor:simple:free}
Let $X$ be any dendro-minimal dendrite with a free arc but not reduced to this arc.

Then the image $\Gb\big/ \Fix(\Ends(X))$ of $\Gb$ in $\Aut(\Free(X))$ is a simple group.\qed
\end{cor}

The proof of Theorem~\ref{thm:simple} uses an idea of Tits from~\cite{Tits70}, in the following setup. Let $I$ be an arc of a dendrite $X$. For any point $t$ of $I$, denote by $X_t(I)$ the preimage of $t$ under the first-point map $X\to I$. Thus, if $X_t(I)$ is not reduced to $\{t\}$, then $t$ is either a branch point or an extremity of $I$. For any space $Y$ and any $t\in Y$, we denote by $\Homeo_t(Y)$ the stabiliser of $t$ in $\Homeo(Y)$.

\begin{lem}\label{lem:P}
In the above notation, suppose we are given $h_t\in\Homeo_t(X_t(I))$ for each $t\in I$. Then there is $h\in\Fix(I)$ such that $h|_{X_t(I)}=h_t$ for all $t\in I$.
\end{lem}

\begin{proof}
Whenever $X_t(I)$ is not reduced to $\{t\}$, it is the closure of the set $X_t(I)\setminus \{t\}$, which is the union of the components of $X\setminus \{t\}$ that do not meet $I$. Therefore, the lemma is a particular case of Lemma~\ref{lem:patchwork}.
\end{proof}

\begin{proof}[Proof of Theorem~\ref{thm:simple}]
Assume that $N$ does not fix $\Ends(X)$ pointwise. Two applications of Lemma~\ref{lem:dm} show that $N$ acts dendro-minimally on $X$ because $N$ is normalised by $\Gb$ which is normal is $G$. Since $\Gb=\Gc$, it suffices to show that every $g\in G$ fixing pointwise a component  $Y$ of $X\setminus\{x\}$ for some $x\in X$ belongs to $N$.

We know from Theorem~10.5 in~\cite{DM_dendrites} that $N$ contains an element $n$ admitting and austro-boreal arc $I=[y,z]$. By Lemma~\ref{lem:dis+}, we can assume, upon conjugating $n$ within $N$, that $I$ lies in $Y$ and that the image $b$ of $x$ under the first-point map to $I$ is some branch point in the interior of $I$.

The proof will be complete if we express $g$ as $g=[h,n]=hnh^{-1}n^{-1}$ for some element $h$ of $\Ga=\Gb$, since then $g=(hnh^{-1})n^{-1}\in N$.

This element is provided by Lemma~\ref{lem:P} as follows. Recall that the $n$-action on $I\setminus \{y,z\}$ is free by definition of austro-boreal arcs. We set $h_t$ to be the identity unless $t=n^k b$ for some $k\geq 0$. In that case, we set $h_{n^k b} = n^k g n^{-k}$, which defines indeed a homeomorphism of  $X_{n^k b}(I)$ fixing $n^k b$. Now the identity $g=[h,n]$ holds by construction, finishing the proof.
\end{proof}

We close this section with an observation about individual homeomorphisms fixing an end point.

\begin{lem}
Let $g$ be a homeomorphism of a dendrite $X$ not reduced to a point.

If $g$ fixes an end point, then $g$ fixes at least two points.
\end{lem}

\begin{proof}
Let $x$ be a $g$-fixed end point and let $y$ be any other point. We denote by $z\in [x,y]$ the point at which $[x,y]$ branches away from $[x, g y]$. Now $g z$ belongs to $[x, g y]$ and thus either $[x, g z]\se [x, z]$ or $[x, z]\se [x, g z]$. Upon replacing $g$ by $g\inv$ and $y$ by $g y$, we may assume that we are in the second case. Then the closure of  $\bigcup_{n\in\mathbf{N}}[x, g^n z]$ is a non-trivial arc whose end points are $g$-fixed.
\end{proof}

\section{Semi-linear orders}
The goal of this section is to relate dendrites to the classical theory of semi-linear orders. This relation being fruitful in both directions, we investigate it slightly beyond what is strictly needed for the results of the introduction. We begin by recalling some standard notions about orders:

Let $(T, \leq)$ be an ordered set. A subset is a \emph{chain} if it is linearly (=totally) ordered. One writes $\da x=\{y: y\leq x\}$ and $\ua x=\{y: y\geq x\}$. An order is \emph{dense} if for all $x,y$ with $x<y$ there is $t$ with $x<t<y$. More generally, a subset $T_0\se T$ is \emph{order-dense in $T$} if for all $x,y\in T$ with $x<y$ there is $t\in T_0$ with $x<t<y$. A subset $T_0\se T$ is \emph{coinitial in $T$} if for all $x\in T$ there is $y\in T_0$ with $y\leq x$.

\begin{defn}
An ordered set $(T, \leq)$ is \emph{(lower) semi-linear} if the following two axioms hold:

\smallskip

$T$ is downwards directed: $\forall x,y \ \exists s : s \leq x,y$;

for all $x\in T$, the set $\da x$ is a chain.

\smallskip
\noindent
Observe that any coinitial subset of a semi-linear order is still semi-linear.
\end{defn}

The definitions are readily checked in the following example:

\begin{example}\label{ex:sl:dendrite}
Let $X$ be a dendrite and $z\in X$ an end of $X$. Then the relation defined on $X$ by
$$x\leq y \Longleftrightarrow [z,x] \se [z,y]$$
is a dense semi-linear order.

More generally, any subset $T\se X$ that is arcwise dense (in the topological sense) is order-dense in $(X, \leq)$ and is a semi-linear order. (Although $T$ is not coinitial in $X$ if $z\notin T$, it is then coinitial in $X\setminus \{z\}$, which is still a semi-linear order.)
\end{example}

The main purpose of this section is to construct and study a functor reverting Example~\ref{ex:sl:dendrite}: to any semi-linear order $T$, we shall associate a ``completion'' $\cha T$ containing it. Under natural assumptions, $\cha T$ will be a dendrite inducing the given order on $T$. See Remark~\ref{rem:sl:compl} for the relation to known completions.

\begin{defn}
Let $(T, \leq)$ be a semi-linear order. We define the ordered set $(\cha T, \leq)$ to be the set of all \emph{full down-chains} of $T$ endowed with the inclusion order (induced from the set of all subsets of $T$). Here a chain $C\se T$ is a \emph{down-chain} if $\da x \se C$ for all $x\in C$. It is called \emph{full} if it contains its supremum or has no supremum in $T$.
\end{defn}

Recall that for a subset $S$ of an ordered set $T$, an element $s\in T$ is a \emph{supremum} of $S$ in $T$ if it is an upper bound for $S$ and if $s\leq s'$ for every upper bound $s'\in T$ for $S$. In particular, a supremum is unique if it exists; it is sometimes denoted by $\vee S$. We should however remember that the definition depends on $T$ since $S$ might not have a supremum anymore in a larger ordered set containing $T$. The infimum is defined dually.

\begin{rem}\label{rem:sl:min}
The empty set has a supremum in a semi-linear order $T$ if and only if $T$ admits a (necessarily unique) minimal element $t$. In that case, $\{t\}$ is a unique minimal element of $\cha T$. On the other hand, $T$ admits no minimum if and only if $\varnothing \in \cha T$. In that case, $\varnothing$ is a unique minimal element of $\cha T$.

Thus, in any case $\cha T$ admits a unique minimum; this will turn out to be our end $z$.
\end{rem}

The definitions imply readily the following.

\begin{lem}\label{lem:sl:embed}
Let $(T, \leq)$ be a semi-linear order. Then the map
$$T\longrightarrow \cha T, \kern3mm x\longmapsto \da x$$
is well-defined, injective, and realises $T$ as a subset of $\cha T$ with the same induced order.\qed
\end{lem}

\begin{lem}
Let $(T, \leq)$ be a semi-linear order. Then $(\cha T, \leq)$ is a semi-linear order too.
\end{lem}

\begin{proof}
The ordered set $\cha T$ is downwards directed since it has a minimum by Remark~\ref{rem:sl:min}. For the other axiom, we pick $C\in \cha T$ and need to check that $\da C$ is a chain in $\cha T$. This holds because the collection of down-chains in any chain is linearly ordered by inclusion.
\end{proof}

Recall that an ordered set is \emph{$\vee$-complete} if every non-empty subset admits a supremum; $\wedge$-completeness is defined dually. An ordered set is \emph{Dedekind-complete} if any non-empty subset that is bounded above admits a supremum.

\begin{prop}\label{prop:sl:compl}
Let $(T, \leq)$ be a semi-linear order. Then $(\cha T, \leq)$ is $\wedge$-complete.

Moreover, any non-empty chain $\sC \se \cha T$ admits a supremum in $\cha T$. In particular, all maximal chains $\sC \se \cha T$ are Dedekind-complete.
\end{prop}

\begin{proof}
Let $\sC \se \cha T$ be a non-empty subset. Then $E=\bigcap_{C\in \sC} C$ is a down-chain in $T$. If $E$ has a supremum $\vee E$ in $T$, the fact that each $C\in\sC$ is a full down-chain implies $\vee E\in C$ and thus $\vee E\in E$. Therefore $E\in\cha T$ and $E$ is an infimum for $\sC$ in $\cha T$.

Suppose that $\sC$ is moreover a chain. Consider $F=\bigcup_{C\in \sC} C$; it is a down-chain in $T$. If $F$ has no supremum in $T$, then $F\in \cha T$ and $F$ is a supremum for $\sC$ in $\cha T$. If on the other hand $F$ has a supremum $\vee F$ in $T$, let $F^*=F\cup \{\vee F\}$. Then $F^*\in \cha T$ and $F^*$ is an upper bound for $\sC$. We claim that $F^*$ is a supremum for $\sC$ in $\cha T$. Indeed, let $F'\in \cha T$ be an upper bound for $\sC$. We need to show $F^* \se F'$ and know already $F \se F'$. Since $F$ is not full, $F \neq F'$. Let thus $t\in F' \setminus F$. Then $t$ is an upper bound for $F$ since $F'$ is a chain and $F'$ a down-chain. Therefore $\vee F \leq t$ and since $F'$ is a down-chain we deduce $\vee F \in F'$ as desired.

Finally, for Dedekind-completeness, it suffices to observe that the supremum and infimum constructed in $T$ must automatically belong to any maximal chain containing $\sC$.
\end{proof}

In the converse direction, we have:

\begin{lem}\label{lem:sl:compl:conv}
Let $(T, \leq)$ be a semi-linear order. Suppose that every non-empty chain in $T$ has a supremum and an infimum in $T$. Then the morphism $T\to\cha T$ is an isomorphism.
\end{lem}

\begin{proof}
We need to prove that any $C\in \cha T$ is of the form $C=\da x$ for some $x\in T$. We begin by noting that $C\neq \varnothing$, or equivalently by Remark~\ref{rem:sl:min}, that $T$ has a minimum. Indeed, the infimum of any maximal chain in $T$ is a minimum for $T$ since $T$ is downwards directed. Being non-empty, $C$ has a supremum $x$ in $T$. Being a full down-chain, $C$ coincides with $\da x$.
\end{proof}

\begin{rem}\label{rem:sl:compl:conv}
The above proof shows that Lemma~\ref{lem:sl:compl:conv} holds under the formally weaker assumption that every non-empty chain in $T$ has a supremum in $T$ and that $T$ has a minimum. 
\end{rem}

Taking into account Proposition~\ref{prop:sl:compl}, we deduce:

\begin{cor}\label{cor:sl:complete}
Let $(T, \leq)$ be a semi-linear order. Suppose that every non-empty chain in $T$ has a supremum in $T$ and that $T$ has a minimum. Then $(T, \leq)$ is $\wedge$-complete.\qed
\end{cor}

Combining again Proposition~\ref{prop:sl:compl} with Lemma~\ref{lem:sl:compl:conv}, we also deduce that the operation $T\mapsto \cha T$ is idempotent.

\begin{cor}\label{cor:sl:idem}
For any semi-linear order $(T, \leq)$, the morphism $\cha T \to \cha{\cha T}$ is an isomorphism.\qed
\end{cor}

\begin{rem}\label{rem:sl:compl}
The above results show that $\cha T$ is a completion of $T$ in some sense. It is larger than the Dedekind-type completion of~\cite[5.3]{Droste_mem}, \cite[1.3]{MaroliPHD} but smaller than the MacNeille completion~\cite{MacNeille}.
\end{rem}

\begin{example}
Consider the subset $T = [-1,0) \cup \{i, -i\}$ of the complex numbers endowed with the order given by the usual order on the real part (but $i$ incomparable to $-i$). This is a dense semi-linear order. Moreover, $\cha T$ can be identified with the closure $T\cup\{0\}$ with the order extended by $0\leq i$, $0\leq -i$. This semi-linear order is however not dense anymore.
\end{example}

The construction of $\cha T$ is functorial in the following sense.

\begin{lem}
Let $(T, \leq)$ be a semi-linear order and $T_0\se T$ a subset. Suppose that $(T_0, \leq)$ is semi-linear (e.g.\ it is coinitial in $T$). Then there is a canonical inclusion of ordered sets $\cha {T_0}\to \cha T$ making the diagram
$$\xymatrix{
T \ar[rr] && {\cha T}\\
T_0  \ar[u]\ar[rr] && {\cha{T_0}}\ar[u]
}$$
commutative.
\end{lem}

\begin{proof}
In view of Lemma~\ref{lem:sl:compl:conv} and Corollary~\ref{cor:sl:idem}, we can simplify the notation by assuming that $T$ admits a supremum and infimum for each of its non-empty chains; we then construct a suitable map $\cha {T_0}\to T$. Namely, we associate to $C\in \cha {T_0}$ its supremum in $T$.

This map is injective. Indeed, if $C, C'\in \cha {T_0}$ admit the same supremum $s$ in $T$, they are both subsets of the chain $\da s\cap T_0$. But full down-chains contained in a same chain in $T_0$ coincide when they have the same supremum in $T$. Finally, it is straightforward that the order induced from $T$ under this map coincides with the order on $T_0$.
\end{proof}

In the next statement, the isomorphism $T\cong \cha T$ is implicitly understood.

\begin{prop}\label{prop:sl:reconstruct}
Let $(T, \leq)$ be a semi-linear order such that every non-empty chain in $T$ has a supremum and an infimum in $T$. Let further $T_0\se T$ be a subset such that $(T_0, \leq)$ is semi-linear.

If $T_0$ is order-dense in $T$, then the canonical map $\cha{T_0}\to T$ is an isomorphism.
\end{prop}

\begin{proof}
We know already that there is a canonical inclusion of semi-linear orders $\cha{T_0}\to T$ defined by mapping $C\in\cha{T_0}$ to its supremum in $T$. We need to prove that it is surjective. Choose thus some $t\in T$. We define $C=\da t \cap T_0$ and observe that the supremum of $C$ in $T$ is $t$ by order-density. It suffices therefore to show that $C$ is a full down-chain in $T_0$; the fact that it is a down-chain is immediate. Suppose therefore that $C$ admits a supremum $s$ in $T_0$. Since $s$ is an upper bound for $C$, it is also an upper bound for the chain $\da t$ by order-density. Thus $s\geq t$. Applying again order-density, we deduce $s=t$. It follows $s\in C$ as desired.
\end{proof}

Contrary to linear orders, partial orders admit a number of unrelated natural topologies, see e.g.~\cite[\S3]{Redfield}. We introduce a topology on an arbitrary semi-linear order $(T, \leq)$ as follows.

\begin{defn}
Let $(T, \leq)$ be a semi-linear order. We endow $T$ with the topology generated by all sets $T \setminus \ua x$ and $\ua x \setminus \{x\}$, where $x$ ranges over $T$.
\end{defn}

We recall that an ordered set is a \emph{$\wedge$-semi-lattice} if any pair $\{p,q\}$ of elements admits an infimum, denoted by $p\wedge q$. This is of course the case if the order is $\wedge$-complete.

We shall say that two points $x,y$ of a topological space are \emph{separated} by a point $t$ if the complement of $t$ can be partitioned into two open sets, each containing one of $x,y$.

\begin{lem}\label{lem:sl:T2}
Let $(T, \leq)$ be a dense semi-linear order which is a $\wedge$-semi-lattice. Then any two distinct points can be separated by a point; in particular, $T$ is Hausdorff.
\end{lem}

\begin{proof}
Let $x,y$ be distinct points in $T$; by symmetry, we can assume $x\nleq y$ without loss of generality. Thus $x\wedge y < x$ and by order-density there is $t$ with $x\wedge y < t< x$. Notice that $t\nleq y$. Now $t$ separates $x$ and $y$ because $x\in \ua t \setminus \{t\}$ and $y\in T \setminus \ua t$.
\end{proof}

\begin{lem}\label{lem:sl:sc}
Let $(T, \leq)$ be a semi-linear order which is a $\wedge$-semi-lattice. If $T$ contains a countable order-dense subset, then $T$ is second countable as a topological space.
\end{lem}

\begin{proof}
Let $T_0\se T$ be a countable order-dense subset. It suffices to verify that the sets $T \setminus \ua t$ and $\ua t \setminus \{t\}$ form a sub-base of the topology when $t$ ranges over $T_0$. Consider first a point $y$ of $T \setminus \ua x$ for an arbitrary $x\in T$. Then $x\wedge y < x$ and therefore there is $t\in T_0$ with $x\wedge y < t< x$. In particular, $y\in T \setminus \ua t$ and $T \setminus \ua t \se T \setminus \ua x$. Next, consider $y\in \ua x \setminus \{x\}$. There is now $t\in T_0$ with $x < t< y$ and $y\in \ua t \setminus \{t\} \se  \ua x \setminus \{x\}$.
\end{proof}

We shall use again the more precise fact established above that the sets $T \setminus \ua t$ and $\ua t \setminus \{t\}$ form a sub-base when $t$ ranges over $T_0$.

\begin{thm}
Let $(T, \leq)$ be semi-linear order admitting a countable order-dense subset. Suppose that every non-empty chain in $T$ has a supremum in $T$ and that $T$ has a minimum.

Then $T$ is a dendrite for the topology that we introduced.
\end{thm}

\begin{proof}
By Remark~\ref{rem:sl:compl:conv}, $T\cong \cha T$. Therefore Proposition~\ref{prop:sl:compl} implies that $(T, \leq)$ is $\wedge$-complete. In particular, $T$ is Hausdorff and second countable by Lemma~\ref{lem:sl:T2} and Lemma~\ref{lem:sl:sc}.

Moreover, every maximal chain $\sC \se T$ is Dedekind-complete by Proposition~\ref{prop:sl:compl}; since moreover $\sC$ admits a countable order-dense subset and has a maximum and a minimum, it is order-isomorphic to a closed interval by Cantor's theorem~\cite[p.~511]{Cantor95}. Our topology is the usual order-topology when restricted to $\sC$ and it follows that every point of $T$ can be connected to the minimum by an arc.

We claim that $T$ is compact. Being second countable, it suffices to find an accumulation point for an arbitrary sequence $(x_n)$ in $T$. Suppose for a contradiction that there is no accumulation point and let $T_0\se T$ be a countable order-dense subset. Upon extracting a subsequence for each $t\in T_0$ in a diagonal process, we can assume that $(x_n)$ has the following property. For every $t\in T_0$ there is an integer $n_t$ such that one of the following holds: either $\forall n\geq n_t: x_n > t$, or $\forall n\geq n_t: x_n \ngeq t$.

Let $C\se T$ be the collection of all $x\in T$ such that $x_n > x$ holds for all but finitely many $n$; in particular $C$ contains all $t\in T_0$ satisfying the first case of the above alternative. We can assume $C\neq\varnothing$ because if $C$ does not contain the minimum $z$ of $T$ then $x_n=z$ for $n$ large enough. The set $C$ is a chain since any two of its elements belong to some $\da{x_n}$. Therefore, $C$ admits a supremum $s$ in $T$ and we proceed to show that $x_n$ converges to $s$ using the sub-base of neighbourhoods determined by $T_0$.

To this end, consider first $t\in T_0$ such that $s\in T \setminus \ua t$. If $x_n$ did not belong to this neighbourhood $T \setminus \ua t$ for a cofinal set of integers $n$, then $\forall n\geq n_t: x_n > t$ by the above alternative. This would imply $t\in C$ and thus $s\geq t$, which is absurd. Consider next $t\in T_0$ such that $s\in \ua t\setminus \{t\}$. Suppose again that $x_n$ is not almost always in $\ua t\setminus \{t\}$; this time, it follows that $\forall n\geq n_t: x_n \ngeq t$. On the other hand, $t$ is not an upper bound for $C$; since $\da s$ is a chain, this implies that there is $x\in C$ with $x>t$. For $n$ large enough, we have $x_n > x > t$, which is absurd.

This completes the proof that $T$ is compact.

At this point, we know in particular that $T$ is connected, compact, Hausdorff and second countable; thus $T$ is a continuum. For continua, one of the equivalent characterisation of dendrites is that any two distinct points of can be separated by a point, see~\cite[10.2]{Nadler}. In our case, this criterion is satisfied by Lemma~\ref{lem:sl:T2}.
\end{proof}

If we start with a dendrite $X$ and order as in Example~\ref{ex:sl:dendrite} by choosing $z\in \Ends(X)$, then the topology that we defined on this ordered set is the original dendrite topology by construction. It follows by naturality of the construction that every automorphism of the ordered set $(X, \leq)$ preserves this topology. Conversely, every homeomorphism of $X$ fixing $z$ preserves the order. We record this as follows.

\begin{prop}
There is a canonical identification $\Homeo_z(X) \cong \Aut(X, \leq)$.\qed
\end{prop}

We shall be more interested in certain countable subsets of $X$ viewed as ordered sets.

\begin{cor}\label{cor:sl:isom}
Let $X$ be a dendrite and $z\in\Ends(X)$. Let $T\se X$ be an arcwise dense subset that is invariant under the stabiliser $\Homeo_z(X)$ of $z$ and endow $T$ with the semi-linear order determined by $z$. Then the natural map
$$\Homeo_z(X) \lra \Aut(T, \leq)$$
is an isomorphism of groups.
\end{cor}

\begin{proof}
The above map is a group homomorphism which is injective because $T$ is dense in $X$; we need to prove that it is onto. By naturality of the construction of $\cha T$ and of the topology, every automorphism of $T$ extends to a homeomorphism of $\cha T$. By Proposition~\ref{prop:sl:reconstruct}, we have a canonical homeomorphism $\cha T \cong X$ and the statement follows.
\end{proof}

\section{Generalized Wa\.zewski dendrites}
This section investigates homogeneity properties of the generalised Wa\.zewski dendrites $D_S$. We start with a general definition.

\medskip
To any finite subset $F$ of a dendrite $X$ we associate a finite vertex-labelled simplicial tree $\langle F \rangle$ as follows. The sub-dendrite $[F]$ is a finite tree in the topological sense, i.e.\ the topological realisation of a finite simplicial tree. Such a simplicial tree is not unique because degree-two vertices can be added or removed without changing the topological realisation. We choose for $\langle F \rangle$ to retain precisely one degree-two vertex for each element of $F$ which is a regular point of the dendrite $[F]$. Thus, $\langle F \rangle$ is a tree whose vertex set contains $F$. Finally, we label the vertices of $\langle F\rangle$ by assigning to each vertex its order in $X$.

We observe that this is a labelling by elements of $\NN \cup \{\infty\}$ and that the labelings that can arise in this way are precisely those which are bounded below by the degree of $x$ in $\langle F \rangle$, i.e.\ by the order of $x$ in $[F]$.

\begin{prop}\label{prop:extension}
Fix $S\se \{3, 4, \ldots, \infty\}$ non-empty. Given two finite subsets $F, F'\se D_S$, any isomorphism of labelled graphs $\langle F \rangle \to \langle F' \rangle$ can be extended to a homeomorphism of $D_S$.
\end{prop}

\begin{proof}
To produce a homeomorphism $h$ of an arbitrary topological space $X$, it suffices to cover $X$ by a finite family $\sY$ of closed subsets and to specify for each $Y\in \sY$ a homeomorphism $h_Y\colon Y \cong h_Y(Y) \se X$ such that $\{h_Y(Y)\}_{Y\in\sY}$ is also a cover by closed sets and with the compatibility conditions
$$h_{Y}|_{Y \cap Z} = h_{Z}|_{Y \cap Z}, \kern 3mm h_Y (Y \cap Z) =  h_Y (Y) \cap h_{Z}(Z)$$
for all $Y, Z\in \sY$. In the case at hand, we use $\langle F\rangle$ to decompose $D_S$ into a family $\sY$ of sub-dendrites, as follows. First, for any adjacent vertices $x,y$ of $\langle F\rangle$, we define $D_S(x,y)$ as the closure of the component of $D_S \setminus\{x,y\}$ containing the interior of the arc $[x,y]$. Next, for each individual vertex $x$ of $\langle F\rangle$, we let $Y_x$ be the union of $\{x\}$ and of all components of $D_S \setminus\{x\}$ that do not meet any of the $D_S(x,y)$ where $y$ ranges over the vertices adjacent to $x$. We take for $\sY$ the collection of all these $D_S(x,y)$ and $Y_x$.

We define likewise the decomposition $\sY'$ associated to $\langle F'\rangle$ and observe that any isomorphism $\langle F \rangle \to \langle F' \rangle$ of labelled graphs induces a bijection $\sH\colon \sY \to \sY'$. Moreover, the intersection of two distinct elements of $\sY$ is either empty or reduces to a single point (which is a vertex of $\langle F\rangle$). Therefore, it suffices to exhibit for each $Y\in\sY$ a homeomorphism $h_Y\colon Y \to \sH(Y)$ which has the prescribed behaviour on the intersection of $Y$ and the vertex set of $\langle F\rangle$.

At this point, we can conclude by the well-known homogeneity properties of $D_S$ as follows. All $D_S(x,y)$ are homeomorphic to $D_S$ itself, and the homeomorphism can be chosen to send $x$ and $y$ to any given pair of distinct ends of $D_S$. This follows from Theorem~6.2 in~\cite{Charatonik-Dilks}, compare also Corollary~3.3 in~\cite{Charatonik95}. As for $Y_x$, it can be written as the union of $D_S(x,x')$ where we choose some $x'\in\Ends(D_S)$ in each component of $D_S \setminus\{x\}$ meeting $Y_x$. The (possibly infinite) number of these components is the difference between the orders of $x$ in $D_S$ and in $[F]$, and therefore it is the same number as for $\sH(Y_x)$. Hence the required homeomorphism $Y_x \cong \sH(Y_x)$ is obtained by patching together the various homeomorphisms on each $D_S(x,x')$. The possibility of infinitely many such components does not raise any continuity issue since they will then form a null-sequence, just as in the proof of Lemma~\ref{lem:patchwork}.
\end{proof}

\begin{rem} The collection of finite subsets of $\Br(D_S)$ is a Fraïssé class. This will be made explicit in a forthcoming work.\end{rem}
The following very special case of Proposition~\ref{prop:extension} in fact hardly different from the ingredient from~\cite{Charatonik-Dilks} that we used in the proof, but we isolate it for further reference:

\begin{cor}\label{cor:DS:2-trans}
For any $S\se \{3, 4, \ldots, \infty\}$ and any $n\in S$, the action of $\Homeo(D_S)$ on $\Br_n(D_S)$, on $\Ends(X)$ and on $\Reg(X)$ is doubly transitive.\qed
\end{cor}

An elementary consequence of double transitivity is the following.

\begin{cor}\label{cor:prod2}
Let $S\se \{3, 4, \ldots, \infty\}$ be a non-empty set and $n\in S\cup\{1,2\}$. Then any element of $\Homeo(D_S)$ is the product of two elements fixing each some point of order $n$ in $D_S$.
\end{cor}

This clarifies the simplicity statement of Corollary~\ref{cor:simple} above:

\begin{cor}\label{cor:DS:simple}
$\Homeo(D_S)$ is a simple group for any non-empty $S\se \{3, 4, \ldots, \infty\}$.
\end{cor}

\begin{proof}
By Corollary~\ref{cor:simple}, $\Gb$ is a simple group, where $G=\Homeo(D_S)$. On the other hand, Corollary~\ref{cor:prod2} implies, in particular, that $G=\Gb$.
\end{proof}

Recalling the basic fact that doubly transitive actions are primitive, we also deduce:

\begin{cor}\label{cor:DS:max}
The stabiliser in $\Homeo(D_S)$ of any point of $D_S$ is maximal as a proper subgroup of $\Homeo(D_S)$.\qed
\end{cor}

Recalling that the set of branch points is always countable whilst both $\Reg(D_S)$ and $\Ends(D_S)$ are uncountable, Corollary~\ref{cor:DS:2-trans} implies also the following.

\begin{cor}\label {cor:stab:den}
The stabilizer of $x\in D_S$ has countable index in $\Homeo(D_S)$ if and only if $x$ is a branch point.\qed
\end{cor}

In the case where $S$ is finite, there are only finitely many labelled trees as above for any given number of vertices. Therefore, we deduce the following from Proposition~\ref{prop:extension}.

\begin{cor}\label{cor:DS:oligo}
If $S$ is finite, then the action of $\Homeo(D_S)$ on $D_S$ (viewed simply as a set) is oligomorphic.\qed
\end{cor}

\noindent
Here \emph{oligomorphic} means that for each $p\in\NN$ the diagonal action of $\Homeo(D_S)$ on $(D_S)^p$ has finitely many orbits~\cite{Cameron_oligo}. When $S$ is infinite, we recover likewise from Proposition~\ref{prop:extension} the weaker fact that $\Homeo(D_S)$ has countably many orbits on $(D_S)^p$, first established in~\cite{Camerlo}.

\medskip
We recall that a topological group is \emph{Roelcke pre-compact} if for every identity neighbourhood $U$ there is a finite set $F$ in $G$ with $G=UFU$. This holds in particular for groups that can be represented as closed oligomorphic permutation groups of countable sets (see e.g.~\cite[\S1.2]{Evans-Tsankov}). Therefore, considering the representation of $\Homeo(D_S)$ into $\Sym(\Br(X))$ as in Proposition~\ref{prop:topol}, we have:

\begin{cor}
If $S$ is finite, then the Polish group $\Homeo(D_S)$ is Roelcke pre-compact.\qed
\end{cor}

By a result of Evans--Tsankov, we can deduce that $\Homeo(D_S)$ has the strong Kazhdan property~(T) as a topological group (which is not the case for all oligomorphic groups, see~\cite[\S6]{Tsankov12}).

\begin{cor}\label{cor:DS:T}
If $S$ is finite, then the Polish group $\Homeo(D_S)$ has the strong Kazhdan property~(T) as a topological group.
\end{cor}

\begin{proof}
In view of Theorem~1.1 in~\cite{Evans-Tsankov} and of the oligomorphic presentation of $\Homeo(D_S)$ into $\Sym(\Br(X))$, it remains only to justify that $\Homeo(D_S)$ has no open subgroup of finite index. This follows from the abstract simplicity of $\Homeo(D_S)$ established in Corollary~\ref{cor:DS:simple}.
\end{proof}

Another application of Proposition~\ref{prop:extension} provides a link with semi-linear orders:

\begin{cor}\label{cor:DS:weak}
Fix $S\se \{3, 4, \ldots, \infty\}$, $n\in S$ and any end $z$ of $D_S$. Then the action of the stabiliser $\Homeo_z(D_S)$ on the subset of pairs
$$\big\{(x,y) : x,y\in\Br_n(D_S), x\neq y \text{ and } x\in [y,z]\big\}$$
is transitive.

In particular, the semi-linear order $(T, \leq)$ induced on $T=\Br_n(D_S)$ as in Example~\ref{ex:sl:dendrite} is weakly two-transitive.
\end{cor}

We recall here that a semi-linear order $(T, \leq)$ is called \emph{weakly two-transitive} if $\Aut(T, \leq)$ acts transitively on the set of pairs $(x,y)$ satisfying $x<y$. 

\begin{proof}[Proof of Corollary~\ref{cor:DS:weak}]
For any such pair $(x,y)$, the labelled tree $\langle \{z,x,y\} \rangle$ has always the same isomorphism type: namely, the vertices $z,x,y$ are aligned and in this order, with labels respectively $1,n,n$. Therefore the statement follows from Proposition~\ref{prop:extension}.
\end{proof}

A group $G$ is said to have \emph{property~(OB)} if every isometric $G$-action on any metric space has bounded orbits. Let us emphasise that in this definition, since no topology on $G$ has been specified, we consider $G$ as a discrete topological group. This property is also called \emph{strong uncountable cofinality}. Amongst equivalent definitions is that every left-invariant metric on $G$ is bounded, see e.g.~\cite[1.2]{Rosendal09}. The above results allow us to leverage a theorem from~\cite{Droste-Truss} on semi-linear orders and deduce:

\begin{thm}\label{thm:OB:stab}
Let $S\se \{3, 4, \ldots, \infty\}$ be a non-empty set and pick $z\in\Ends(D_S)$.

Then the group $\Homeo_z(D_S)$ has property~(OB).
\end{thm}

\begin{proof}
Choose any $n\in S$ and let $T=\Br_n(D_S)$. We endow $D_S$ and $T$ with the semi-linear order determined by $z$ as in Example~\ref{ex:sl:dendrite}. By Corollary~\ref{cor:DS:weak}, the order $(T, \leq)$ is weakly two-transitive. By Theorem~3.1 in~\cite{Droste-Truss}, it follows that the group $\Aut(T, \leq)$ has property~(OB). Finally, Corollary~\ref{cor:sl:isom}, states that $\Homeo_z(D_S)$ is isomorphic to $\Aut(T, \leq)$.
\end{proof}

\begin{cor}\label{cor:OB}
For any non-empty $S\se \{3, 4, \ldots, \infty\}$, the group $\Homeo(D_S)$ has property~(OB).
\end{cor}

\begin{proof}
Let $z$ be an end of $D_S$. By Theorem~\ref{thm:OB:stab},  $H=\Homeo_z(D_S)$ has property~(OB). By Corollary~\ref{cor:DS:2-trans} the action of $G=\Homeo(D_S)$ on $G/H$ is doubly transitive. Now Lemma~\ref{lem:OB} below completes the proof.
\end{proof}

\begin{lem}\label{lem:OB}
Let $G$ be a group and $H<G$ a subgroup such that the $G$-action on $G/H$ is doubly transitive. If $H$ has property~(OB), then so does $G$.

The same holds more generally if $H$ has the \emph{relative} property~(OB) in $G$, as defined in~\cite{Rosendal14_arx}.
\end{lem}

\begin{proof}
The assumption on $G/H$ is equivalent to the fact that we have $G=H \cup HgH$ for some $g\in G$; hence the statement follows from the characterisation in terms of left-invariant metrics.
\end{proof}

Another application of the isomorphism afforded by Corollary~\ref{cor:sl:isom} between end stabilisers in $\Homeo(D_S)$ and $\Aut(T, \leq)$ will be useful in Section~\ref{sec:non-i}, namely:

\begin{prop}\label{prop:DS:index}
Let $S\se \{3, 4, \ldots, \infty\}$ be non-empty and let $z\in\Ends(D_S)$. Let $M\lhd \Homeo_z(D_S)$ be the normal subgroup of elements fixing pointwise some non-trivial arc containing $z$.

Then $M$ contains all proper normal subgroups of $\Homeo_z(D_S)$.

Furthermore, $\Homeo_z(D_S)$ has no proper subgroup of finite index.
\end{prop}

\begin{proof}
As before we fix $n\in S$ and consider the countable semi-linear order $(T, \leq)$ for $T=\Br_n(D_S)$ as in Example~\ref{ex:sl:dendrite}, recalling that it is weakly two-transitive by Corollary~\ref{cor:DS:weak}. Moreover, $H=Homeo_z(D_S)$ is canonically isomorphic to $\Aut(T, \leq)$ by Corollary~\ref{cor:sl:isom}. Therefore, we can apply Theorem~1.3 in~\cite{Droste-Holland-Macpherson} which holds for all countable weakly two-transitive semi-linear orders. This result implies that every proper normal subgroup $N\lhd \Aut(T)$ is contained in the normal subgroup $R(T)\lhd \Aut(T)$ of elements that fix pointwise $\{y:y< t\}$ for some $t\in T$. As a subgroup of $\Homeo(D_S)$, this is precisely the group $M$.

For the additional statement, it suffices to prove that the group $H/M$ is infinite since every finite index subgroup contains a normal finite index subgroup. Let $g\in \Homeo(D_S)$ be a homeomorphism admitting an austro-boreal arc ending at $z$; in particular, $g\in H$. Such an element $g$ exists by Lemma~\ref{lem:obarc} below. We claim that $g\notin M$. This claim then also holds for any non-trivial power of $g$ and hence the group generated by the image of $g$ in $H/M$ is indeed infinite.

To prove the claim, it suffices to observe that any arc terminating at $z$ must meet $I$ at more than just $z$ because $z$ is an end; therefore the claim follows since the only $g$-fixed points in $I$ are the two extremities of $I$. 
\end{proof}

\begin{lem}\label{lem:obarc}
Any non-trivial arc of $D_S$ is austro-boreal for some element of $\Homeo(D_S)$. 
\end{lem}

\begin{proof}
Let $[x,y]$ be a non-trivial arc of $D_S$. By Theorem~10.5 in~\cite{DM_dendrites}, we know that there is an element $g'$ of $\Homeo(D_S)$ admitting some austro-boreal arc $[x', y']$. By~\cite[6.2]{Charatonik-Dilks}, both $D_S(x,y)$ and $D_S(x',y')$ are homeomorphic to $D_S$ and hence to each other (these sub-dendrites were defined in the proof of Proposition~\ref{prop:extension}). We can now define $g\in \Homeo(D_S)$ by transporting the $g'$-action from $D_S(x',y')$ to $D_S(x,y)$ and letting $g$ act trivially outside $D_S(x,y)$.
\end{proof}

\section{Non-isomorphic homeomorphism groups}\label{sec:non-i}
Consider any non-empty set $S\se \{3, 4, \ldots, \infty\}$. The initial goal of this section is to prove that the homeomorphism group of the generalised Wa\.zewski dendrite $D_S$, as an abstract group, determines the set $S$. The proof will show that the space $D_S$ can be recovered from the mutual positions of the stabilisers of branch points as abstract subgroups of $\Homeo_x(D_S)$.

Given a point $x\in D_S$, denote by $\sC_x$ the set of connected components of $D_S\setminus\{x\}$. For each $C\in\sC_x$, the closure $\overline C = C\cup\{x\}$ is homeomorphic to $D_S$ by Lemma~\ref{lem:sub:DS}. The stabiliser $\Homeo_x(D_S)$ has a natural representation to the permutation group $\Sym(\sC_x)$. Moreover, this representation is split surjective and yields a permutational wreath product:

\begin{lem}\label{lem:unscrewingstab}
$$\Homeo_x(D_S) \cong \left(\prod_{C\in\sC_x}\Homeo_x(\overline C)\right)\rtimes\Sym\left(\sC_x\right).$$
\end{lem}

\begin{proof}
The kernel of the representation to $\Sym(\sC_x)$ is all of $\prod_{C\in\sC_x}\Homeo_x(\overline C)$ because of Lemma~\ref{lem:patchwork}. Therefore it suffices to prove that there is a subgroup of $\Homeo_x(D_S)$ mapping isomorphically onto $\Sym(\sC_x)$ under this representation.

To this end, we select an end $z$ of $D_S$ and choose for each $C\in \sC$ a homeomorphism $\varphi_C\colon\overline{C}\to D_S$ such that $\varphi_C(x)=z$. Given a permutation $\sigma\in\Sym(\sC_x)$, we obtain $\widetilde\sigma\in \Homeo_x(D_S)$ by defining $\widetilde\sigma$ to be $\varphi_{\sigma(C)}^{-1}\circ\varphi_C$ on $\overline C$ for each $C\in\sC_x$. The fact that $\widetilde\sigma$ is indeed a homeomorphism even if $\sC$ is infinite follows from the fact that $\sC$ is a null-family, as in the proof of Lemma~\ref{lem:patchwork}. Now $\sigma\mapsto\widetilde\sigma$ is indeed a section of $\Sym(\sC_x)$.
\end{proof}

\begin{cor}\label{cor:unique:index}
If $x\in D_S$ is a point of finite order, then every finite index subgroup of $\Homeo_x(D_S)$ contains the kernel of the representation onto $\Sym(\sC_x)$.
\end{cor}

\begin{proof}
Since $\overline C\cong D_S$ and since $\sC_x$ is finite, Proposition~\ref{prop:DS:index} implies that this kernel, namely the product $\prod_{C\in\sC_x}\Homeo_x(\overline C)$, has no finite index (proper) subgroup.
\end{proof}

\begin{thm}\label{thm:strong:isom}
Let $S,S'\se\{3, 4, \ldots, \infty\}$ be two non-empty subsets. Suppose that there is a group isomorphism
$$\Phi\colon \Homeo(D_S) \lra \Homeo(D_{S'}).$$
Then there is a map $\fhi\colon D_S\to D_{S'}$ which is a $\Phi$-equivariant homeomorphism.
\end{thm}

In particular, Theorem~\ref{thm:strong:isom} shows that $S$ is determined by the group $\Homeo(D_S)$.

\begin{cor}\label{cor:non-iso}
The groups $\Homeo(D_S)$ and $\Homeo(D_{S'})$ are isomorphic if and only if $S=S'$.\qed
\end{cor}

Another consequence of Theorem~\ref{thm:strong:isom} is the following.

\begin{cor}\label{cor:outer}
The group $\Homeo(D_S)$ has no outer automorphisms.
\end{cor}

\begin{proof}[Proof of Corollary~\ref{cor:outer}]
Let $\Phi$ be an automorphism of the group $\Homeo(D_S)$. Then Theorem~\ref{thm:strong:isom} provides $\fhi\in \Homeo(D_S)$ which is $\Phi$-equivariant. This means by definition that for all $x\in D_S$ and all $g\in \Homeo(D_S)$ we have $\fhi(g x) = \Phi(g) \fhi(x)$. In other words, $\Phi(g) = \fhi g \fhi\inv$; this shows that $\Phi$ is inner.
\end{proof}

\begin{proof}[Proof of Theorem~\ref{thm:strong:isom}]
Throughout the proof, we consider $\Homeo(D_S)$ as acting on $D_S$ but also on $D_{S'}$ via $\Phi$; we only mention $\Phi$ explicitly when a confusion could occur. We begin with the following claim:

For every branch point $x\in D_S$, the stabiliser $\Homeo_x(D_S)$ fixes a branch point in $D_{S'}$; moreover this point is unique.

Indeed, consider the decomposition of $\Homeo_x(D_S)$ given by Lemma~\ref{lem:unscrewingstab} and let $n\in S$ be the order of $x$. Since $\overline C\cong D_S$, Proposition~\ref{prop:DS:index} shows that we can apply Proposition~\ref{prop:perm} and deduce that $\Homeo_x(D_S)$ fixes a point $y$ in $D_{S'}$. We recall that all point stabilisers are maximal by Corollary~\ref{cor:DS:max}. Therefore, since $\Phi$ is an isomorphism, we deduce that the image of $\Homeo_x(D_S)$ is exactly $\Homeo_y(D_{S'})$. Now $y$ must be a branch point in view of Corollary~\ref{cor:stab:den}. The uniqueness follows from the maximality of stabilisers, since otherwise $D_S$ would have two branch points $y, y'$ with the same stabilisers; this is readily seen to contradict the double transitivity of Corollary~\ref{cor:DS:2-trans}, using that $\Br_m(y)$ is arcwise dense, where $m$ is the order of $y$ in $D_{S'}$.

Next, we claim that the order of $y$ is also $n$. It suffices to prove that $n$ is an invariant of the abstract group $H=\Homeo_x(D_S)$ amongst stabilisers of branch points. Indeed, the case $n=\infty$ is characterised as the only case when $H$ has no index two subgroup thanks to Lemma~\ref{lem:no:ind2}. When $n\geq 3$ is finite, it can be recovered from the index of the minimal finite index subgroup given by Corollary~\ref{cor:unique:index}.

At this point, we can already define a natural map $\fhi\colon \Br(D_S) \to  \Br(D_{S'})$ that preserves the Menger--Urysohn order. This map is $\Phi$-equivariant by construction and bijective because it is natural in $\Phi$. We note that we have incidentally already $S=S'$. Since $\Br(D_S)$ is dense in $D_S$, it suffices now to show that $\fhi$ is uniformly continuous with respect to the uniform structure on $\Br(D_S)$ induced by $D_S$. This then implies that $\fhi$ has a continuous extension to $D_S$. Such a continuous extension is automatically an equivariant homeomorphism.

By Lemma~\ref{lem:unif}, this uniform structure is generated by the entourages $U_F$ (restricted to $\Br(D_S)$), where $F$ ranges over the finite subsets of $\Br(D_S)$. Therefore, it suffices to prove that the map $\fhi$ preserves the ternary relation on $\Br(D_S)$ given by $x\in [y,z]$ for $x,y,z\in \Br(D_S)$.

We first show that $\fhi$ preserves the following weaker ternary relation: $\{x,y,z\}$ are contained in a common arc. It suffices to express this relation purely in terms the stabilizers of branch points. This can be done because $\{x,y,z\}$ are \emph{not} contained in a common arc if and only if there is $w \neq x,y,z$ such that
$$ \Homeo_w(D_S) \supseteq \Homeo_x(D_S) \cap \Homeo_y(D_S)\cap \Homeo_z(D_S).$$
The ``only if'' direction holds by considering the centre $w$ of a tripod $(x,y,z$). The ``if'' direction follows readily from the transitivity properties of Proposition~\ref{prop:extension}.

It only remains to express the relation $x\in [y,z]$ in terms of the above weaker ternary relation on $\Br(D_S)$. This is done as follows: $x\in [y,z]$ if and only if for all $w\in \Br(D_S)$, either $\{w,x,y\}$ or $\{w,x,z\}$ are contained in a common arc. The ``only if'' direction holds by definition. For the ``if'' direction, suppose that $x\notin [y,z]$. Then the arc $I$ connecting $x$ to its first-point projection to $[y,z]$ is not reduced to a point. Therefore, we can choose $w\in \Br(D_S)$ which does not lie in any arc containing $I$, and thus neither $\{w,x,y\}$ nor $\{w,x,z\}$ are contained in a common arc.
\end{proof}

Theorem~\ref{thm:strong:isom} can also be deduced from Rubin's much more general (and correspondingly more difficult) Theorem~0.2 in~\cite{Rubin89}. Indeed:

\begin{prop}\label{prop:*}
For any non-empty set $S\se\{3, 4, \ldots, \infty\}$, the pair $(D_S,\Homeo(D_S))$ satisfies Rubin's condition~(\textasteriskcentered).
\end{prop}

Recall here that a pair $(X,G)$, where  $X$ is a Hausdorff topological space and $G$ is a subgroup of $\Homeo(X)$, satisfies Rubin's condition~(\textasteriskcentered) from~\cite{Rubin89} if the following hold:

\begin{enumerate}
\item $(X,G)$ is \emph{regionally disrigid}: for any non-empty open subset $U\se X$, there is a non-trivial $g\in G$ such that $g$ is the identity on $X\setminus U$.
\item Any non-empty open subset $U\se X$ contains a non-empty open subset $U_1\se U$ which is \emph{flexible} with respect to $G$, i.e. for any open subsets $V,W\se U_1$ with $gV\cap W\neq\emptyset$ for some $g\in G$, there is $g'\in G$ such that  $g' V\cap W\neq\emptyset$ and $g'$ is the identity on $X\setminus U$.
\end{enumerate}

\begin{proof}
We claim first that every non-empty open subset $U\se D_S$ contains an open subset $U_1$ homeomorphic to $D_S\setminus \{z\}$, where $z\in\Ends(D_S)$. Indeed, by density of $\Ends(D_S)$, we can choose an end $z_0\in U$. Since $z_0$ is an end, it admits a neighbourhood $U_1\se U$ whose topological boundary in $D_S$ is reduced to a single point $z$ (compare e.g.~\cite[9.3]{Nadler}). Since the closure of $U_1$ is homeomorphic to $D_S$ by Lemma~\ref{lem:sub:DS}, the claim follows.

It suffices to prove that the disrigidity and flexibility conditions hold for such a set $U_1$. In both cases, we can work directly with the action of the stabiliser $\Homeo_z(D_S)$ on $D_S$, then transport the resulting homeomorphisms to $U_1$ and extend them by the identity on $D_S\setminus U_1$. In that setting, both conditions are immediate consequences of the fact that $\Homeo_z(D_S)$ acts transitively on $\Br_n(D_S)$ (Proposition~\ref{prop:extension} or Corollary~\ref{cor:DS:weak}) and that $\Br_n(D_S)$ is dense in $D_S$.
\end{proof}

We expect that general dendro-minimal dendrites cannot be reconstructed from their homeomorphism groups. However, besides the case of the generalised Wa\.zewski dendrites, this reconstruction is also possible as soon as there is a free arc:

\begin{thm}\label{thm:rec:free}
Let $X$ be a dendro-minimal dendrite with a free arc. If $G=\Homeo(X)$ is isomorphic, as a group, to $G'=\Homeo(X')$ for any dendro-minimal dendrite $X'$, then $X'\cong X$.

Moreover, the isomorphism of groups is induced by such a homeomorphism; therefore, $G$ has no outer automorphisms.
\end{thm}

Notice that Theorem~\ref{thm:rec:free} does not, at first sight, fit into Rubin's setting since $X$ and $X'$ are not assumed to belong both to the same class (of dendro-minimal dendrites with free arcs). But in fact, the first step of the proof is to establish that the group isomorphism forces $X'$ to admit a free arc; from then on, one can invoke Rubin's theorem.

\begin{proof}[Proof of Theorem~\ref{thm:rec:free}]
Let $I$ be a maximal free arc of $X$. Combining the automatic continuity of~\cite{Rosendal-Solecki} with the extreme amenability of~\cite{Pestov98}, it follows that any group action of $\Homeo_+(I)$ on any compact metrisable space has a fixed point, see Corollary~7 in~\cite{Rosendal-Solecki}. In particular, it follows that $\Homeo_+(I)$ fixes a point in $X'$.

Next, we recall from Proposition~\ref{prop:fix:stab} that $G$ contains the product $\prod_{I\in \Free(X)} \Homeo_+(I)$ as a normal subgroup. Lemma~\ref{lem:fixprodgen} implies that this product group fixes a point in $X'$. Since it is a normal subgroup and since $X'$ is dendro-minimal, Lemma~\ref{lem:dm} implies that this product fixes $\Ends(X')$ pointwise. Recalling that the set of ends is dense unless $X'$ admits a free arc, we conclude that $X'$ does indeed admit a free arc.

As mentioned above, this is a point where Rubin's theorem can be applied. Indeed, the condition~(\textasteriskcentered) can be readily verified after observing that any non-empty open set contains some free arc, which follows from Proposition~\ref{prop:free:dense}.
\end{proof}

\bibliographystyle{halpha}
\bibliography{dendrite}
\end{document}